\documentclass[12pt]{article}
\usepackage[pctex32]{graphics}

\usepackage{amssymb,latexsym,amsmath,epsfig,amsthm}
\newtheorem{theorem}{Theorem}

\begin{document}
	\begin{center}
		\uppercase{\bf The problem with two linear branches}
		\vskip 20pt
		{\bf Fritz Schweiger}\\
		 Department of Mathematics, University of Salzburg, Salzburg, Austria\\
		{\tt fritz.schweiger@plus.ac.at}\\
	\end{center}
	
		\vskip 30pt
	{\bf Abstract}
	\noindent 
	Piecewise fractional linear maps wzth three or more branches have been studied in several papers. For many Moebius maps the shape of the density of their invariant measurs can be written down exactly. However, if just two branches are linear no explicit form is known. In this paper a partial solution is offered.

	
	{\bf Introduction}\\
	Since Rényi's foundational work \cite{R} ergodic theory connected to number theory has seen a lot of publications (see  \cite{Sch95}, \cite{IoKr}, \cite{Sch16}).  
	Fractional linear maps  $T: [0,1] \to [0,1]$
	with three branches are interesting examples of fibred sytems in ergodic theory (see \cite{Sch06},\cite{Sch18}). Many fractional linear systems with three branches are ergodic and admit an invariant measure  $\mu$ absolutely continuous to Lenesgue measure $	\lambda$ However, less is known about the density of $\mu$. The basic method is the use of a {\em dual map}. It is explained as follows.\\
	We consider {\em Moebius maps} with the partition $0 < p_1 < p_2 < 1$ and three matrices $V_{\alpha}, V_{\beta}. V_\gamma$ such that the associated fractional linear maps satisfy $V_\alpha[0,p_1] = V_\beta[p_1, p_2] = V_\gamma [p_2,1] = [0,1]$. Then we use the maps associated to the transposed matrices  $V_{\alpha}^{*}, V_{\beta}^{*}, V_{\gamma}^{*}$. \\
	(1)  There is an invertible symmetric matrix $M$ such that $M V_i = V_i^{*}M$, $i = \alpha, \beta,  \gamma$. The image $B^{*} = M[0.1] =[\rho, \sigma]$ together with these three maps is called a {\em natural dual} and 
$$  h(x)= \int_{B^{*}} \frac{dy_1 dy_2 }{(1+ x_1 y_1 + x_2 y_2)^3} $$ is the invariant density. \\
(2) It can happen that there exists an interval $\rho, \sigma]$ such that the three branches $V_{\alpha}^{*}, V_{\beta}^{*}, V_{\gamma}^{*} $  form a fibred system which is not a natural dual. It is called an {\em exceptional dual}. Then again 
$$  h(x)= \int_{B^{*}} \frac{dy_1 dy_2 }{(1+ x_1 y_1 + x_2 y_2)^3} $$ is the invariant density. If there is a common fixed point $\tau$ for all branches then 
$$ h(x)= \frac{1}{(1+ x \tau)^2}.$$
(3)	If there is fibred system with two branches $V_0$ and $V_1$ such that $V_\alpha = V_0, V_\beta = V_1 V_0
, V_\gamma = V_1 V_1$ or $V_\alpha = V_0, V_\beta = V_1 V_1, V_\gamma = V_1V_0$ then the density can be found using  jump transformation (this device is called {\em 1-step extension}, see \cite{Sch25}). There exist two other cases with $V_\gamma = V_1$. \\
Now if $Vx = \frac{a+bx}{c}$, which means that $V$ is linear, then  $V^{*}y = \frac {by}{c+ay}$  and $V^{*}0 = 0$. This observation leads to an interesting problem. If just one of the three branches is linear then there are good chances for finding the density of the invariant measure. If all three branches are linear then the density is $h(x) =1$. If two branches are linear and just one branch is non-linear, then the shape of the density of the invariant measure (if it exists) is unknown! This paper gives a partial solution  to this problem. \\

{\bf The first construction}\\
 
Here we use the partition $0 < \frac{1}{3}
< \frac{2}{3} < 1$. Let $T$ be a fractional linear map  such that $T$ is linear on $[0, \frac{1}3]$ and on $[\frac{2}{3},1]$ with $T[0, \frac{1}3] = T[\frac{2}{3}, 1]= [\frac{1}3, \frac{2}{3}]$ and non-linear on $[\frac{1}{3}, \frac{2}{3}]$ and $T[\frac{1}{3}, \frac{2}{3}]= [0, 1]$. Later we will see that the value $\beta = 0$ corresponds to $T$ linear too. \\
We consider a jump transformation $S$, namely  the Moebius map  defined by $S = T^2$ on $[0, \frac{1}{3}] \cup [\frac{2}{3}, 1]$ and $S=T$ on $[\frac{1}{3}, \frac{2}{3}]$.
 If $\nu$ is an invariant measure for $S$ with density $h$  then $\mu$ (with density $g$),  defined by $\mu(E)= \nu (E)+ \nu(T^{-1}E \cap [0,\frac{1}{3}]) + \nu(T^{-1}E \cap  [\frac{2}{3},1])$ is, invariant for $T$.\\
 Since $S$ is a piecewise fractional linear map wth three branches we try to find antural dual. If $V_\alpha, V_\beta, V_\gamma$ are the three inverse branches of $T$ then $V_{\alpha \beta}, V_\beta, V_{\gamma \beta}$ are the three inverse branches of $S$. Note that the density $g$ is given as
 $$g(x) = h(x), , 0 < x < \frac{1}{3} , \frac{2}{3} < x < 1$$
 $$g(x)= h(x)+ g(V_\alpha x)\omega_\alpha (x) + h(V_\gamma x)\omega_\gamma (x), \, \frac{1}{3} < x <  \frac{2}{3}  .$$
Here $\omega_\alpha$ and $\omega_\gamma$ denote the Jacobians of $V_\alpha$ and $V_\gamma$.\\
 The result coincides with the result for $\beta = 0$, namely $h(x)=	$ and $g(x)= 3$ on $[\frac{1}{3},  \frac{2}{3}].$ \\
 We write $\varepsilon = 1 $ to indicate an increasing branch and $\varepsilon =-1$ to indicate a decreasing branch. The {\em type} of T is denoted $[\varepsilon_\alpha, \varepsilon_\beta, \varepsilon_\gamma]$. Note, that we assume $-1 < \beta \leq 2$. However,  $\beta = 0$ hbelongs to  the linear map..\\

\begin{theorem} . 
	(1) If the map $T$ is of type $[1,1,1]$, then  $S$ has a natural dual and 
	$$ h(x) = \frac{1}{(2 - \beta + 3 \beta x) (2 + 3 \beta x )}.$$
	(1) If the map $T$ is of type $[1,-1,1]$, then  $S$ has a natural dual and 
$$ h(x) = \frac{1}{(4+ \beta + 3 \beta x) (4 + 3 \beta x )}.$$
\end{theorem}

	\begin{proof}
	We list the relevant (inverse) branches of $T$ and $S$.
	$$ \varepsilon_\alpha =1: V_\alpha x = \frac{-1+3x}{3}  , , \varepsilon_\alpha =-1: V_\alpha x = \frac{2 - 3x}{3} $$          
		$$ \varepsilon_\beta =1: V_\beta x = \frac{1 + (1+2 \beta)x}{3 + 3 \beta x}, \, \varepsilon_\beta = -1 : V_\beta x = \frac{2 + (-1 + \beta )x}{3 + 3\beta x}$$
		$$ \varepsilon_\gamma =1: V_\gamma x = \frac{1+3x}{3}  , , \varepsilon_\gamma =-1: V_\gamma x = \frac{4 - 3x}{3}           $$
	To determine the map $S$ one calculates  $V_{\alpha \beta}x = V_\alpha (V_\beta x)$
	and $V_{\gamma \beta}x = V_\gamma (V_\beta x)$.\\	
	We obtain
	$$ \varepsilon_{\alpha\beta} = 1: V_{\alpha\beta} x = \frac{(1+\beta)x}{3 + 3 \beta x}, \,\varepsilon_{\alpha\beta} = -1: V_{\alpha\beta} x = \frac{(1 - x}{3 + 3 \beta x}$$
		$$ \varepsilon_{\gamma\beta} = 1: V_{\gamma\beta} x = \frac{2+(1+3\beta)x}{3 + 3 \beta x}, \,\varepsilon_{\gamma\beta} = -1: V_{\gamma\beta} x = \frac{3 +(-1+2\beta)  x}{3 + 3 \beta x}$$
	 We look for a map $Mx = \frac{B + D x}{A+ Bx}$ such that $M (V_i x) = V_i^{*} (Mx)$, $ i \in \{\alpha \beta, \beta, \gamma \beta\}$.
	Then we obtain three homogeneous equations for three indeterminants $A, B, D$. A non-trivial  soltion requires that the Determinant vanishes,  $DET = 0$.\\ 
	(1)  Type $[1,1,1]$\\
	We find the homogeneous equations
	$$ A(3\beta) + B(-2 + \beta) = 0$$
	$$ A 3\beta + B(-2 + 2\beta) - D =0$$
	$$A3\beta + B (-2 + 3\beta) -2D = 0.$$
	Then $DET = 0$ and we find $A= 2- \beta$, $B= 3\beta$, and $D=3\beta^2$.\\
	As an example we take $\beta = 2$. Then we have $$V_{\alpha \beta}x  = \frac{x}{1+2x}, V_\beta x = \frac{1+5x}{3+6x}, V_{\gamma \beta}x = \frac{2+7x}{3+6x} $$ and $$h(x) = \frac{1}{x(1+3x)}.$$
	
	(2) Type $[1,-1,1]$\\
	Again we find $DET = 0$ and the values $A= 4 + \beta$, $B= 3 \beta$, and $D= 3 \beta^2$.	
		As an example we take $\beta = 1$. Then we have $$V_{\alpha \beta}x  = \frac{1-x}{3+x}, V_\beta x = \frac{2}{3+3x}, V_{\gamma \beta}x = \frac{3+x}{3+3x} $$ and $$h(x) = \frac{1}{(4+3x)(5+3x)}.$$
		
	\end{proof}
	\begin{theorem}  
		The map $S$ has no natural dual for the  other 6 types of $T$.
	\end{theorem}
	
	\begin{proof}
			Since the map $\psi x = 1-x$ is an isomorphism for the pairs $[1,-1-1]
	$, $[-1,-1,1]$ and $[1,1,-1]$, $[-1,,1]$ just 4 cases remain. Therefore we give the list of remaining 4 types and their determinants.\\
	$[1,-1,-1]: 	DET = 6 \beta$\\
	$[1,1,-1]: 		DET = 6\beta^2$\\
	$[-1,1,-1]: 	DET = 12\beta(1+\beta)	$\\
	$[-1,-1,-1]:	DET =  9\beta(2+\beta)$
	
\end{proof}

{\bf A general result}\\

We started with $p_1 =\frac{1}{3}$ and $p_2 = \frac{2}{3}$.  This choice was a very lucky one. The next theorem is a little bit surprising.

\begin{theorem}
	Let $0 < p_1 < p_2 < 1$, then for type $[1,1,1]$ and for type $[1,-1,1]$      of $T$ the map $S$ has a natural dual if and only if $p_1^2 + p_2^2 - p_1 p_2 - p_1 = 0$.
	\end{theorem}
\begin{proof}
We give the relevant maps.
	$$ \varepsilon_\alpha =1: V_\alpha x = \frac{-p_1^2 + p_1x}{p_2 - p_1}   $$          
$$ \varepsilon_\beta =1: V_\beta x = \frac{p_1 + (p_2 - p_1 +p_2\beta)x}{1 +  \beta x}, \, \varepsilon_\beta = -1 : V_\beta x= \frac{p_2 + (-p_2 + p_1 + p_1\beta )x}{1 + \beta x}$$
$$ \varepsilon_\gamma =1: V_\gamma x = \frac{p_2^2 -p_1 + (1-p_2)x}{p_2 - p_1}  , , \  $$	For type $[1,1,1]$ we calculate
$$V_{\alpha\beta} x = \frac{(p_1+p_1\beta)x}{1 + \beta x}, \, V_{\gamma\beta} x = \frac{p_2+(1 - p_2+\beta)x}{1 +  \beta x}$$
and find $DET = (p_1^2 + p_2^2 - p_1 p_2 - p_1)(\beta^2 + \beta)$.\\
For type $[1,-1,1]$ we calculate
 $$V_{\alpha\beta} x = \frac{p_1 - p_1x}{1 +  \beta x}, \,  V_{\gamma\beta} x = \frac{1 +(-1 + p_2 + p_2\beta)  x}{ 1+ \beta x}$$
 and find $DET = (p_1^2 + p_2^2 - p_1 p_2 - p_1)(\beta^2 + \beta)$.
In both cases $DET = 0$ only if $\beta =0$, $\beta = -1$ (not an allowed value) or $p_1^2 + p_2^2 - p_1 p_2 - p_1 = 0$.
\end{proof}
If $p_1 + p_2 =1$ then we find $p_1 = \frac{1}{3}$ and $p_2 = \frac{2}{3}$. Another rational solution is  $p_1 = \frac{1}{7}$ and $p_2 = \frac{3}{7}$.\\

	{\bf Final remark}\\
	
	The paper is restricted to find natural duals of $S$. There could exist  exceptional duals but the great number of possible configurations could be an interesting further work.


\begin{thebibliography}{99}
		\bibitem{IoKr} Iosifescu, M. and Kraaikamp, C. 2002: {\em Metrical Theory of Continued Fractions} Kluwer Axademic Publishers, Dordrecht( Boston(London 2002))2002
		\bibitem{R} Rényi, A. 1957: Representations for real numbers and their ergodic properties {\em Acta Math. Acad. Sci. Hunggar.} 8, 477-493 
		\bibitem{Sch95} Schweiger, F. 1995: {\em Ergodic Theory of Fibred Systems and Netric Number Theory} Okford Science Publications, New York The Clarendon Press, Oxford 1995	
		\bibitem{Sch06} Schweiger, F.  Differentiable equivalence of fractional linear maps.
		{\em IMS Lecture Notes-Monograph Series Dynamics \& Stochastics} 48
		(2006), 237-247
		\bibitem{Sch16} Schweiger, F. 2016: {\em Continued fractions and their   generalizations. A short history of f-expansions}. Docent Press, Boston 2016
	 \bibitem{Sch18} Schweiger, F. 2018: Invariant measures of piecewise fractional linear maps and piecewise quadratic maps. {\em Int. J Number Theory} 14 (218), 1559-1572
	 \bibitem{Sch25} Schweiger, F. 2025: . On Moebius maps which are characterized by the configuration of their dual maps 
	 {\em Indag. Math.} 16 (2025). 806-818
	 
	 
	\end{thebibliography}
\end{document}